\documentclass[11pt,a4paper,twoside]{article}
\usepackage{a4}
\usepackage{color}
\usepackage{bm, amsmath, amssymb, amsthm} 
\usepackage{graphicx}

\def\ddx{\partial^2_{xx}\/}
\def\<{\langle}
\def\>{\rangle}
\def\eps{\varepsilon}

\def\ZZ{\mathbb{Z}}
\def\RR{\mathbb{R}}

\def\calf{\mathcal{F}}

\def\tr{\operatorname{Tr\,}}
\def\id{\operatorname{id\,}}
\def\Div{\operatorname{div}}

\def\vol{\operatorname{vol}}

\def\eq{\hspace*{-1mm}&=&\hspace*{-1mm}}
\def\plus{\hspace*{-1mm}&+&\hspace*{-1mm}}

\def\dt{\partial_t}

\newtheorem{corollary}{Corollary}
\newtheorem{definition}{Definition}
\newtheorem{example}{Example}
\newtheorem{remark}{Remark}
\newtheorem{lemma}{Lemma}
\newtheorem{proposition}{Proposition}
\newtheorem{theorem}{Theorem}

\author{Vladimir Rovenski
\thanks{
        Mathematical Department, University of Haifa,
        E-mail: rovenski@math.haifa.ac.il}
   \ and \ Robert Wolak\thanks{
        Faculty of Mathematics and Computer Science, The Institute of Mathematics of Jagiellonian University, Krakow,
        E-mail: robert.wolak@im.uj.edu.pl}}
\title{Deforming metrics of foliations}

\begin{document}

\date{}

\maketitle


\begin{abstract}
We study geometry of a manifold endowed with two complementary orthogonal distributi\-ons 
(plane fields)
and a time-dependent Riemannian metric.
The~work begins with formulae concerning deformations of geometric quantities as the metric varies conformally along one of the distributions.
Then we introduce the geometric flow depending on the mean curvature vector of the distribution,
and show existence/uniquenes and convergence of a solution as $t\to\infty$,
when the complementary distribution is integrable with compact leaves.
We~apply the method to the problem of prescribing mean curvature vector field of a foliation,
and give examples for harmonic and umbilical foliations and for the double-twisted product metrics,
including the codimension-one case.


\end{abstract}

\section{Introduction}
\label{sec:1-intro}

\textit{Geometric Flows}  (GFs) are important in many fields of mathematics and  physics.
A~GF is an evolution of a geometric structure under a differential equation related to a functional on a manifold, usually associated with some curvature.
The most popular GFs in mathematics are the \textit{heat flow}
(\cite{jj2011}, \cite{ML1951} etc), the \textit{Ricci flow} (\cite{bre}, \cite{to} etc)
 and the \textit{mean curvature flow}.
They all correspond to dynamical systems in the infinite-dimensional space of all appropriate geometric structures
on a given manifold.
GF equations are quite difficult to solve in all generality, because of nonlinearity. Although the short time existence of solutions is guaranteed by the parabolic or hyperbolic nature of the equations, their (long time) convergence to canonical geometric structures is analyzed under various conditions.

\textit{Extrinsic geometry} of foliated Riemannian manifolds describes properties which can be expressed in terms of the second fundamental form of the leaves
and its invariants (mean curvature vector, higher mean curvatures and so~on).
One of the principal problems of extrinsic geometry of foliations reads as follows:
 \textit{Given a foliation $\calf$ on a manifold $M$ and an extrinsic geometric property (P), does there exist a Riemannian metric $g$ on $M$ such that $\calf$ enjoys (P) w.\,r.\,t.~$g$?}

Such problems (first posed by H.\,Gluck in 1979 for geodesic foliations) were studied already in the 1970's when D.\,Sullivan provided a topological condition (called {\it topological tautness}) for a foliation, equivalent to the existence of a Riemannian metric making all the leaves minimal (see \cite{cc1}).
 In~recent decades, several tools providing results of this sort have been developed. Among them, one may find
 {\it foliated cycles} \cite{su1} and
 new \textit{integral formulae}, \cite[Part 1]{rw-m}, \cite{rw0}, \cite{rov2}, \cite{wa1},
the very first of which is Reeb's vanishing of the integral of the mean curvature.
  Few works consider GFs on foliated manifolds, see \cite{A-LK2001}, \cite{LMR}, \cite{nrt}).
A GF on a foliated manifold is \textit{extrinsic}, if the evolution depends on the second fundamental tensor of the (leaves of the) foliation.
 Recently, the first author and P.\,Walczak introduced and studied \textit{extrinsic GF}s
 on codimension-one foliations, see  \cite[Parts 2 and 3]{rw-m}.
The~paper extends this field of research for foliations of arbitrary codimension.

Let $(M, g)$ be a connected closed Riemannian manifold with a pair of complementary orthogonal distributions $D$ and ${D^\perp}$
(plane fields) of dimensions $n$ and $p$, respectively.
If $D$ (or $D^\perp$) is integrable, it is tangent to a foliation $\calf$ (respectively, $\calf^\perp$).
Denote by
$b$ and ${b^\perp}$ the second fundamental tensors of $D$ and ${D^\perp}$ (w.\,r.\,t. to $g$),
$H=\tr_{\!g} b\in\Gamma(D^\perp)$ and ${H^\perp}=\tr_{\!g}({b^\perp})\in\Gamma(D)$ their mean curvature vectors.
 We~call $D$ (or $D^\perp$) \textit{umbilical}, \textit{harmonic} or \textit{totally geodesic} if
\[
 n\,b=H\cdot g_{\,|D},\ \
 H=0\ \ {\rm or} \ \ b=0\quad
 ({\rm respectively}, \ p\,b^\perp=H^\perp\cdot g_{\,|D^\perp},\ \ H^\perp=0 \ \ {\rm or} \ \ b^\perp=0).
\]
By means of the natural representation of the structure group ${\rm O}(p)\times {\rm O}(n)$ on $TM$,
Naveira \cite{N1983} obtained thirty-six distinguished classes of Riemannian almost-product manifolds.
Following this line of research, several geometers, \cite{GM1983}, \cite{M1984} and \cite{M1983}, completed the geometric interpretation, gave nontrivial examples for each class, and studied the behavior of the different conditions (and hence of the different classes of almost-product structures) under a conformal change of the metric.

 The notion of the $D$-\emph{truncated} $(r,k)$-\emph{tensor field} $\hat S$
(where $r=0,1$, and $\ \widehat{}\,$ denotes the $D$-component)  will be helpful:
 $\hat S(X_1,\dots,X_k) = S(\hat X_1,\dots,\hat X_k)\ (X_i\in TM)$.
We introduce the \textit{Extrinsic Geometric Flow} (EGF) as a family  $g_t$ of Riemannian metrics on $M$ satisfying
the PDE
\begin{equation}\label{eq1}
 \dt g_t = -\frac2n\,(\Div^\perp\!H_t)\,\hat g_t,
\end{equation}
where $g_0=g$ is given.
Here $\hat g$ is the $D$-truncated metric tensor $g$ on $M$,
i.e., $\hat g(X,Y)=g(X,Y)$ and $\hat g(\xi,\cdot)=0$ for all $X,Y\in D$ and $\xi\in D^\perp$.
Some operations w.\,r.\,t. EGF metrics $g_t$ restricted on $D^\perp$ are $t$-independent
(e.g., $\nabla^\perp$ and $\Div^\perp$).
Given a Riemannian metric $g$, the mean curvature vector $H$
can be expressed in terms of the first partial derivatives of~$g$.
Therefore $g\mapsto\Div^\perp H$ is a second-order partial differential operator.

We show that EGFs serve as a tool for studying the following:

\vskip1mm
\textbf{Question:} \emph{Under what conditions on $(M,D)$ the EGF metrics $g_t$ converge to one for which $D$
enjoys a given extrinsic geometric property (P), e.g., is harmonic, umbilical, or totally geodesic\,?}

\vskip1mm
The structure of the paper is as follows. Section~\ref{sec:main-res} collects main results.
Section~\ref{subsec:tvar} develops formulae for the deformation of
extrinsic geometric quantities of a foliation as the Riemannian metric varies conformally along $D$.
Section~\ref{sec:1-egf} introduces EGF and shows the existence/uniquenes and converging of a global solution $g_t\ (t\ge0)$. In particular, is is verified that

 (i) the 1-form $\theta_H$ (dual to $H$) satisfies the heat equation along $D^\perp$.

 (ii) the metrics $g_t$ preserve the ''umbilical'' (''totally geodesic'', etc.) property of $D$.

 (iii) for appropriate $g_0$ the metrics $g_t$ converge to a metric $g_{\infty}$ with ``harmonic distribution $D$".

\vskip.5mm\noindent
In Theorem~\ref{T-02}, the orthogonal distribution $D^\perp$ is integrable and the method of proof
is based on solving the heat equation for 1-forms on the leaves of $D^\perp$.
In Theorem~\ref{T-03} the method is applied to the problem of prescribing mean curvature vector field of $D$.
Section~\ref{sec:more} contains results and examples for codimension-one foliations (Proposition~\ref{C-2ntau1}), and double-twisted product metrics.
Appendix (Section~\ref{sec:1-prelim}) collects the necessary facts about the heat equation and the heat flow for 1-forms.

\section{Main Results}
\label{sec:main-res}

We define the connection $\nabla^\perp$ induced on~$D^\perp$ by $\nabla^\perp_X\xi =(\nabla_X \xi)^\perp$ (i.e., $\nabla_X \xi $ is projected onto $D^\perp$),
where $\xi\in D^\perp$ and $X\in TM$.
 In this section we suppose that \textit{${D^\perp}$ is integrable with all leaves compact and orientable}.
(If both $D$ and $D^\perp$ are integrable and the foliation tangent to $D^\perp$ is totally geodesic
then the foliation tangent to $D$ is \textit{Riemannian}).

 Denote $\theta_\xi$ the 1-form on $D^\perp$ dual to the vector field $\xi\in\Gamma(D^\perp)$.
The 1-form $\theta_\xi$ is $D^\perp$-harmonic (see Section~\ref{sec:2.4-heat1})
if and only if
 $\delta^\perp\theta_\xi=0$ (i.e., $\Div^\perp\!\xi=0$)
 and
 $d^\perp\theta_\xi=0$ (i.e., $\nabla^\perp\xi$ is symmetric).

\begin{theorem}\label{T-02}
Let the leaves of ${D^\perp}$ compose a fibration $L^\perp\overset{i}\hookrightarrow M\overset{\pi}\to B$.
 If $\,\dim D^\perp=p>1$, suppose that $d^\perp\theta_{H}=0$.
Then (\ref{eq1}) admits a unique smooth solution $g_t$ for all $t\ge0$
that converges in $C^\infty$-topology as $t\to\infty$ to a Riemannian metric $g_\infty$
for which $D$ is harmonic.
\end{theorem}

\begin{example}\rm
Let the leaves of integrable distribution $D^\perp$ be flat tori $T^p$.
Any differential $1$-form on $T^p$ can be written as $\omega=\sum_i\omega_i dx^i$.
The form $\omega$ is harmonic if and only if the functions $\omega_i$ are harmonic, and therefore constant.
In this case, the vector field dual to $\omega$ is constant.
Indeed, the space of constant vector fields on a flat torus $T^p$ is isomorphic to $H^1(T^p,\RR)\simeq\RR^p$.
\end{example}

A foliation, whose normal plane field is umbilical, is locally conformally equivalent to a Riemannian foliation, see \cite{M1983}. The next corollary completes this fact.

\begin{corollary}\label{C-umbilic}
Under the assumptions of Theorem~\ref{T-02}, let  $D$ be $g_0$-umbilical.
Then $D$ is $g_\infty$-totally geodesic ($g_\infty$ is the limit of metrics (\ref{eq1}) as~$t\to\infty$).
\end{corollary}

A foliation with vanishing mean curvature is called \textit{harmonic}.
Every leaf of such foliation is a minimal submanifold of $M$.
A foliation $\calf$ is \textit{taut} if there is at least one metric on $M$ for which $\calf$ is harmonic.
In particular, if there is an immersed closed transversal manifold that intersects each leaf, then $\calf$ is taut
(see \cite{rsw} and survey in \cite{cc1}). For example, a Reeb foliation on $S^3$ is not taut.
 The known proofs of existence of  ``taut" metrics use the Hahn--Banach Theorem and are not constructive.
Corollary~\ref{C-minimal} (of Theorem~\ref{T-02}) shows how to produce  in some cases
a family of metrics converging to the metric for which $\calf$ is harmonic (i.e., $H_\infty=0$).
Certain results for codimension-one foliations are given in Section~\ref{sec:more}.

\begin{corollary}\label{C-minimal}
Let $\calf $ be a foliation on $(M,g)$ tangent to $D$ of codimension $p>1$.
Suppose that the leaves of ${D^\perp}$ compose a fibration $L^\perp\overset{i}\hookrightarrow M\overset{\pi}\to B$,
and the equality $d^\perp\theta_{H_0}=0$ is satisfied.
Then (\ref{eq1}) admits a unique solution $g_t\ (t\ge0)$,
converging in $C^\infty$-topology as $t\to\infty$ to a Riemannian metric $g_\infty$, for which $\calf$ is harmonic.
\end{corollary}

Let $\calf$ be a foliation of any codimension of a closed manifold $M$ and $X$
be a vector field on $M$. Recently, P. Schweitzer and P. Walczak \cite{sw2004} provided some necessary and sufficient conditions for $X$ to become the mean curvature vector of $\calf$ with respect to some Riemannian metric on $M$.

\vskip1mm
Extending the definition of EGF and method of Theorem~\ref{T-02}, we show how to produce in some cases (e.g., Theorem~\ref{T-03}) a one-parameter family of metrics converging to the metric with prescribed mean curvature vector field of $\calf$.

\begin{theorem}\label{T-03}
Let $p>1$ and the leaves of ${D^\perp}$ compose a fibration $L^\perp\overset{i}\hookrightarrow M\overset{\pi}\to B$.
Then for any smooth vector field $X$ on $M$ orthogonal to $\calf $ and satisfying
$d^\perp\theta_{H-X}=0$, the PDE
\begin{equation}\label{eq1-X}
 \dt g_t=-\frac 2n\Div^\perp(H_t - X)\,\hat g_t,\qquad g_0=g
\end{equation}
admits a unique solution $g_t\ (t\ge0)$, converging in $C^\infty$-topology as $t\to\infty$ to a Riemannian metric $g_\infty$ with $D^\perp$-harmonic 1-form $\theta_{H_\infty-X}$.
If $H^1(L^\perp,\RR)=0$ for the leaves of $D^\perp$, then $H_\infty=X$.
\end{theorem}

\begin{example}\rm
(a) By Bochner theorem (see \cite{jj2011}), if the Ricci curvature of any leaf $L^\perp$ of $D^\perp$ is
non-negative everywhere and positive for some point then $H^1(L^\perp,\RR)=0$ (see Theorem~\ref{T-03}).

\vskip1mm
(b) The following example (communicated to authors by P. Walczak) shows us that the condition $d^\perp\theta_{H-X}=0$
(see Theorem~\ref{T-03}) and the assumption $d^\perp\theta_{H}=0$ (see Theorem~\ref{T-02} and Corollary~\ref{C-minimal}) are needed. Let $X$ be a divergence free (e.g., a Killing) vector field on the leaves $S^p$ of the product $M=M_1\times S^p$ of a unit $p$-sphere and a Riemannian manifold $(M_1,g_1)$. Let the distribution $D$ on $M$ corresponds to $TM_1$.
Then the product metric $g$ has the mean curvature $H=0$, and $g$ is a fixed point of the dynamical system (\ref{eq1-X}).
Consequently, $H_t = 0$ for all $t\ge0$ and $H_\infty = 0\ne X$.
\end{example}

\section{$D$-conformal variations of geometric quantities}
\label{subsec:tvar}

In this section we develop formulae for deformations of geometric quantities as the Riemannian metric varies conformally along one of the distributions.

\subsection{Preliminaries}
\label{subsec:prel}

Denote by $\mathcal{M}$ the space of smooth Riemannian metrics of finite volume on $M$ such that ${D^\perp}$ is orthogonal to $D$.
Elements of $\mathcal{M}$ are called $(D,D^\perp)$-\textit{adapted metrics}.
Let $\mathcal{M}_1\subset\mathcal{M}$ be the subspace of $(D,D^\perp)$-adapted metrics of unit volume, and
 $\pi:\mathcal{M}\to \mathcal{M}_1$,
 where $\pi(g)=\bar g=\big(\vol(M,g)^{-2/n}\,\hat g\big)\oplus g^\perp$
be the $D$-conformal projection.
Let $g_t\in\mathcal{M}$ (with $0\le t<\eps$) be a family of metrics with $g_0$ of unit volume.
Consider the $D$-truncated tensor field
\[
 S_t=\dt g_t.
\]
Recall that the \textit{musical isomorphism}\index{isomorphism (musical)!$\sharp$} $\sharp:T^*M\to TM$ sends a covector
$\omega=\omega_i dx^i$ to $\omega^\sharp =\omega^i\partial_i=g^{ij}\omega_j\partial_i$,
and $\flat: TM\to T^*M$\index{isomorphism (musical)!$\flat$} sends a vector $X=X^i\partial_i$ to $X^\flat=X_idx^i=g_{ij}X^jdx^i$.
We denote by $S^\sharp$ the $(1,1)$-tensor field on $M$ which is $g$-dual to a symmetric $(0,2)$-tensor $S$,
$$
 S(X,Y)=g(S^\sharp(X), Y)\quad \mbox{\rm for all vectors } \ X,Y.
$$
The volume form $\vol_t$ of $g_t$ evolves as
 $\frac{d}{dt}\vol_t=\frac12\,(\tr S^\sharp)\vol_t$, see~\cite{to}.
If $S_t=s_t\,\hat g_t$ with $s_t:M\to\RR$ (i.e., $g_t$ are conformally equivalent along $D$
and $\hat g_t$ is the $D$-truncated metric $g_t$) then
\begin{equation}\label{E-dtdvol}
 \frac{d}{dt}\vol_t=\frac n2\,s_t\vol_t.
\end{equation}
Hence, the metrics $\tilde g_t=(\phi_t\hat g_t)\oplus g^\perp_t$
with dilating factors $\phi_t=\vol(M,g_t)^{-2/n}$, belong to $\mathcal{M}_1$.

 Recall that the \textit{Levi-Civita connection} $\nabla^t$ of a metric $g_t$ on $M$ is given by
\begin{eqnarray}\label{eqlevicivita}
\nonumber
 2\,g_t(\nabla^t_X Y, Z) \eq X(g_t(Y,Z)) + Y(g_t(X,Z)) - Z(g_t(X,Y)) \\
 \plus g_t([X, Y], Z) - g_t([X, Z], Y) - g_t([Y, Z], X)
\end{eqnarray}
for all vector fields $X, Y$ and $Z$ on $M$.
 Since the difference of two connections is always a tensor, $\Pi_t:=\dt\nabla^t$ is a $(1,2)$-tensor field on $(M, g_t)$.
Differentiation (\ref{eqlevicivita}) w.\,r.\,t. $t$ yields the formula, see~\cite{to},
\begin{equation}\label{eq2}
 g_t(\Pi_t(X, Y), Z)=\frac12\big[(\nabla^t_X S_t)(Y,Z)+(\nabla^t_Y S_t)(X,Z)-(\nabla^t_Z S_t)(X,Y)\big]
\end{equation}
for all $X,Y,Z\in TM$.
Indeed, if the vector fields $X=X(t),\,Y=Y(t)$ are $t$-dependent, then
\begin{equation}\label{eq2time}
 \dt\nabla^t_X Y = \Pi_t(X, Y) + \nabla_X(\dt Y)+ \nabla_{\dt X} Y.
\end{equation}
Notice the symmetry $\Pi_t(X, Y)=\Pi_t(Y, X)$ of the tensor $\Pi_t$.

We will use the following condition for convergence of evolving metrics (see \cite[Appendix~A]{bre}).

\begin{proposition}\label{P-converge}
 Let $\dt g_t=s_t\,\hat g_t\ (t\ge0)$ be a one-parameter family of
 Riemannian metrics on a closed manifold $M$ with complementary distributions $D$ and $D^\perp$.
 Define functions $u_m(t)=\sup_M |(\nabla^t)^m s_t|_{g(t)}$ and assume that $\int_0^\infty u_m(t)\,dt<\infty$ for all $m\ge0$.
 Then, as $t\to\infty$, the metrics $g_t$ converge in $C^\infty$-topology to a smooth Riemannian metric $g_\infty$.
\end{proposition}

\begin{proof} Our assumptions ensure that $g_t$ converge in $C^\infty$-topology to a symmetric $(0,2)$-tensor $g_\infty$.
The metrics are uniformly equivalent: $c^{-1}\hat g_0\le \hat g_t\le c\,\hat g_0$ for some $c>0$ and all $t\ge0$.
Hence, $g_\infty$ is positive definite.
\end{proof}

Let $\nabla^\perp\phi$ be the $D^\perp$-component of the gradient of a function $\phi\in C^1(M)$.
The second fundamental tensor of $D$ (similarly of ${D^\perp}$) is defined by
\begin{equation}\label{E-def-bT}
 b(X,Y)=\frac12\,(\nabla_X Y+\nabla_Y X)^\perp,\qquad X,Y\in D.
\end{equation}

The second fundamental forms of $D$ with respect to metrics $g$ and $\tilde g= (e^{\,2\/\phi}\hat g)\oplus g^\perp$
are related by the following lemma.

\begin{lemma}[see \cite{rw-m} for codimension-one foliations]\label{L-btAt}
Let $(M,\,g=\hat g\oplus g^\perp)$ be a Riemannian manifold with complementary orthogonal distributions $D$ and ${D^\perp}$.
Given $\phi\in C^1(M)$, define a metric $\tilde g = (e^{\,2\/\phi}\hat g)\oplus g^\perp$.
Then the second fundamental forms and the mean curvature vectors of $D$ w.\,r.\,t. $\tilde g$ and $g$ are related by
\begin{equation}\label{E-Ak-U}
 \tilde b = e^{\,2\/\phi}\big(b -(\nabla^\perp\phi)\,\hat g\big),\qquad
 \tilde H = H -(\dim D)\nabla^\perp\phi.
\end{equation}
So, if
$\phi$ is constant then, by (\ref{E-Ak-U}), we have:
$\tilde b=e^{\,2\/\phi}\,b\,$ and $\tilde H=H$.
 \end{lemma}

\begin{proof} By (\ref{eqlevicivita}), for any $X,Y\in D$ and $\xi\in D^\perp$ we have
\[
 g(\tilde\nabla_X Y,\xi) = e^{\,2\/\phi}\,g(\nabla_X Y,\xi) -e^{\,2\/\phi}\,g(X,Y)\,\xi(\phi)
 -\frac12\,(e^{\,2\/\phi}-1)\,g([X,Y], \xi).
\]
From this and definition (\ref{E-def-bT}), formula (\ref{E-Ak-U})$_1$ follows.
Since $H=\tr_{\!g}\,b$, we have (\ref{E-Ak-U})$_2$.
\end{proof}

\subsection{The integral formula}
\label{subsec:tensors}

The \textit{divergence of a vector field} $X$ on $(M,g)$ is given by
 $\Div X=\sum\nolimits_{s} g(\nabla_{e_s} X, e_s)$,
where $(e_s)$ is a local orthonormal frame on $(M,g)$.
Recall the identity for a smooth function $f:M\to\RR$,
\begin{equation*}
  \Div(f\cdot X)=f\cdot\Div X + X(f).
\end{equation*}
For a compact manifold $M$ with boundary and inner normal $n$, the \textit{Divergence Theorem} reads as
\begin{equation}\label{E-Th-div}
 \int_M \Div X\,d\vol=\int_{\partial M} g(X,n)\,d\omega.
\end{equation}
For a closed manifold $M$, we have $\int_M \Div X\,d\vol=0$.
 The $D^\perp$-\textit{divergence}, $\Div^\perp \xi$, of a vector field $\xi\in\Gamma({D^\perp})$ is defined similarly
 to $\Div\xi$, using a local orthonormal frame $(\eps_\alpha)$ of~${D^\perp}$.

\begin{lemma}\label{L-intDH}
For a vector field $\xi\in\Gamma({D^\perp})$ on a closed manifold $M$, we have the identity
\begin{equation}\label{E-divN}
 \int_M (\Div^\perp \xi)\,d\vol =\int_M g(H, \xi)\,d\vol.
\end{equation}
So, $\int_M (\Div^\perp \xi)\,d\vol=0$ for any vector field $\xi\in\Gamma({D^\perp})$ if and only if $H=0$.
\end{lemma}

\begin{proof}
Using the definition $H=\sum\nolimits_{i\le p} b(e_i, e_i)$, we have
\begin{equation*}
 \Div\xi-\Div^\perp\xi =\sum\nolimits_{i\le p} g(\nabla_{e_i}\xi, e_i) =-\sum\nolimits_{i\le p} g(b(e_i, e_i), \xi) =-g(H,\xi).
\end{equation*}
By the Divergence Theorem, $\int_M \Div\,\xi\,d\vol=0$, we obtain (\ref{E-divN}).
\end{proof}

\begin{remark}\rm
(i) By Lemma~\ref{L-intDH} with $\xi=H$, we have
\begin{equation}\label{E-divHt}
 \int_M (\Div^\perp H)\,d\vol =\int_M g(H, H)\,d\vol\ge0.
\end{equation}

(ii) By Lemma~\ref{L-intDH} with $\xi=\nabla^\perp f$, for a function $f\in C^2(M)$, we have
\begin{equation*}
 \int_M (\Delta^\perp f)\,d\vol =\int_M g(\nabla f, H)\,d\vol.
\end{equation*}
Here $\Delta^\perp f=\Div^\perp(\nabla^\perp f)$ is the $D^\perp$-\textit{Laplacian} of $f$.

(iii)
In analogy with the fact that \textit{on a closed connected Riemannian manifold, every harmonic function (i.e., $\Delta f=0$) is constant},
we claim:
\textit{If $\Delta^\perp f=g(\nabla f, H)$ then $\nabla^\perp f=0$}.
Indeed,
 \[
 \Div(f\nabla^\perp f) + f(H(f)-\Delta^\perp f) = g(\nabla^\perp f,\nabla^\perp f).
 \]
Using the Divergence Theorem, we obtain $\int_M g(\nabla^\perp f,\nabla^\perp f)\,d\vol=0$, and then $\nabla^\perp f=0$.
\end{remark}

\subsection{$D$-related geometric quantities}
\label{subsec:tvarb}

Let $\{e_i,\,\eps_\alpha\}\ (i\le n,\,\alpha\le p)$ be a local $g_0$-orthonormal frame on $TM$ adapted to $D$ and ${D^\perp}$.

\begin{lemma}\label{prop-Ei-a}
 Let $\{e_i\}$ be a local $g_0$-orthonormal frame of $\,D$ (on a set $U_q\subset M$), and $\dt g_t = s\,\hat g_t$.
 Suppose that $\{e_i(t)\}$ evolves according to
 \begin{equation}\label{E-frameE}
 \dt e_i(t)=-\frac12\,s\,e_i(t).
\end{equation}
 Then $\{e_i(t)\}$ is a $g_t$-orthonormal frame of $D$ on $U_q$ for all $t$.
\end{lemma}

\proof We have
\begin{eqnarray*}
 \dt(g_t(e_i, e_j))=g_t(\dt e_i(t), e_j(t)) +g_t(e_i(t), \dt e_j(t)) +(\dt g_t)(e_i(t), e_j(t))\\
 =s\,\hat g_t(e_i(t), e_j(t))-\frac12\, g_t(s\,e_i(t), e_j(t)) -\frac12\, g_t(e_i(t), s\,e_j(t))=0.\qed
\end{eqnarray*}

The following lemma is compatible with Lemma~\ref{L-btAt}.

\begin{lemma}\label{L-btAt2}
Let $g_t\in\mathcal{M}$ and $\dt g_{t}=s_{t}\,\hat g_{t}$ for some $s_{t}\in C^1(M)$.
Then the second fundamental tensor $b$, its mean curvature vector $H$ and the $D^\perp$-divergence $\Div^\perp H$ are evolved by
\begin{eqnarray}\label{E-S-b}
 \dt b(X,Y)\eq s\,b(X,Y)-\frac12\,\hat g(X,Y)\,\nabla^\perp s
 +\frac12\,(\nabla_{X} \dt Y +\nabla_{\dt X} Y+\nabla_{Y} \dt X+\nabla_{\dt Y} X)^\perp,\\
\label{E-S-H}
 \dt H \eq -\frac n2\,\nabla^\perp s,\qquad
 \dt (\Div^\perp H) = -\frac{n}2\,\Delta^\perp s,\qquad
 \dt\,\theta_H = -\frac{n}2\,d^\perp s.
\end{eqnarray}
\end{lemma}

\begin{proof}
Let $S=\dt g_{t}$ be $D$-truncated.
By (\ref{eq2}), (\ref{E-def-bT}), symmetry of $S$ and $S(\cdot, {D^\perp})=0$, we have
\begin{eqnarray*}
 g_t(\dt b(X,Y), \xi) \eq\frac12\,g_t\big(\dt(\nabla^t_X Y)+\dt(\nabla^t_Y X), \,\xi\big)\\
\nonumber
 \eq \frac12\big[(\nabla^t_X S)(Y,\xi)+(\nabla^t_Y S)(X,\xi) -(\nabla^t_{\xi} S)(X,Y)\big] +Q
\end{eqnarray*}
for all  $\xi\in{D^\perp}$ and $t$-dependent $X,Y\in D$.
Here,
 $Q:=\frac12\,g_t(\nabla^t_{X} \dt Y +\nabla^t_{\dt X} Y+\nabla^t_{Y} \dt X+\nabla^t_{\dt Y} X,\,\xi )$
due to (\ref{eq2time}).
Substituting $S=s\,\hat g$, we obtain the required (\ref{E-S-b}):
\begin{eqnarray*}
 g(\dt b(X,Y), \,\xi)\eq -\frac12\big[s\,\hat g(Y,\nabla^t_X\xi)+s\,\hat g(X,\nabla^t_Y \xi)+\xi(s)\hat g(X,Y)\big] +Q\\
 \eq s\,g(b(X,Y),\,\xi)-\frac12\,\hat g(X,Y)\,\xi(s)+Q.
\end{eqnarray*}
Let $\{e_i(t)\}$ be a local $g_t$-orthonormal frame of $D$ on $U_q$ for all $t$, hence (\ref{E-frameE}) holds,
see Lemma~\ref{prop-Ei-a}. By the above we obtain (\ref{E-S-H})$_1$ (see also alternative proof in Remark~\ref{R-altprooh}):
\begin{eqnarray*}
 \dt H \eq \sum\nolimits_i\dt b(e_i(t),e_i(t)) \\
 \eq\sum\nolimits_i\big[s\,b(e_i(t),e_i(t)) -\frac12\,g_t(e_i(t),e_i(t))\,\nabla^\perp s\big]
 -\sum\nolimits_i s\,b(e_i(t), e_i(t))=-\frac n2\,\nabla^\perp s.
\end{eqnarray*}
 To show (\ref{E-S-H})$_2$, let $S=\dt g_{t}$ be $D$-truncated (i.e., $S(D^\perp,\cdot)=0$).
 Using (\ref{eq2}), we have
\begin{eqnarray*}
 \dt (\Div^\perp H)\eq\sum\nolimits_{\alpha\le p}\dt (g(\nabla_{\eps_\alpha} H, \eps_\alpha))
 =\sum\nolimits_{\alpha\le p}\big[ (\dt g)(\nabla_{\eps_\alpha} H, \eps_\alpha) +g_t(\dt(\nabla_{\eps_\alpha} H), \eps_\alpha)\big]\\
 \eq \sum\nolimits_{\alpha\le p}\big[ S(\nabla_{\eps_\alpha} H, \eps_\alpha)
 +\frac12(\nabla_H S)(\eps_\alpha, \eps_\alpha)\big] +\Div^\perp(\dt H)
 =\Div^\perp(\dt H).
\end{eqnarray*}
Here, we used
$S(\eps_\alpha, \cdot)=0$ and
 $(\nabla_H S)(\eps_\alpha, \eps_\alpha)=H(S(\eps_\alpha, \eps_\alpha))-2\,S(\nabla_H\,\eps_\alpha, \eps_\alpha)=0$.
Now, assuming $S=s\,\hat g$, and using (\ref{E-S-H})$_1$, we obtain (\ref{E-S-H})$_2$:
$\dt (\Div^\perp H)=\Div^\perp(\dt H)=-\frac n2\,\Div^\perp(\nabla^\perp s)=-\frac{n}2\,\Delta^\perp s$.

To show (\ref{E-S-H})$_3$, we use (\ref{E-S-H})$_1$ to calculate for any $X\in D^\perp$:
\[
 \dt\theta_H(X) = \dt(g_t(H_t,X))= s\,\hat g(H_t,X)+g_t(\dt H_t, X)+g_t(H_t, \dt X)
 =-\frac n2\,g_t(\nabla^\perp s, X)+g_t(H_t, \dt X).
\]
Hence, $(\dt\theta_H)(X)
=-\frac n2\,(\nabla^\perp s)^\flat(X)
+g_t(H_t, \dt X)-\theta_H(\dt X)=-\frac n2\,d^\perp s(X)$.
\end{proof}

\begin{remark}\label{R-altprooh}\rm The alternative proof of (\ref{E-S-H})$_1$ is based on the identity
(see \cite[Lemma~2.4]{rw-m} for $k = 0$)
\begin{equation*}
 \dt(\tr_{\!g_t} B) =\tr_{\!g_t}(\dt B) -\<B, S\>_{g_t},
\end{equation*}
where $S = \dt g$, $B$ -- a $t$-dependent symmetric $(k, 2)$-tensor on $(M,g)$, and $\<B, S\>=B^{ij}S_{ij}$.

In our case, $k = 1$, $B = b$, $S = s\,\hat g_{t}$ and $\tr_{\!g_t} B = H_{t}$. Thus, using (\ref{E-S-b}), we have
\begin{eqnarray*}
 \dt(\tr_{\!g_t} B)\eq\dt H,\qquad
 \tr_{\!g_t}(\dt B)=\tr_{\!g_t}(\dt b) = s H-\frac n2\nabla^\perp s,\\
 \<B, S\>_{g_t}\eq\<b, s\,\hat g_t\>_{g_t} =s\<b_t, \hat g_t\>_{g_t} =s\tr_{\!g_t} b_t =s\,H_t.
\end{eqnarray*}
\end{remark}

Next we show that $D$-conformal variations of metrics preserve the umbilicity of~$D$.

\begin{proposition}\label{P-08b}
Let $\dt g_t=s_t\,\hat g_t\ (s_t:M\to\RR)$, be a $D$-conformal family of Riemannian metrics on a manifold $(M,D,{D^\perp})$. If $D$ is umbilical for $g_0$, then $D$ is umbilical for any $g_t$.
\end{proposition}

\begin{proof} Since $D$ is $g_0$-umbilical, we have $b=\frac1n\,H\,\hat g$ at $t=0$, where $H$ is the mean curvature vector field of $D$.
Applying to (\ref{E-S-b}) the theorem on existence/uniqueness of a solution of ODEs,
we conclude that $b_t=\frac1n\,\tilde H_t\,\hat g_t$ for all $t$, for some $\tilde H_t\in\Gamma(D^\perp)$.
Tracing this, we see that $\tilde H_t$ is the mean curvature vector of $b_t$, hence
$D$ is umbilical for any $g_t$.
\end{proof}

\subsection{${D^\perp}$-related geometric quantities}
\label{subsec:tvarbbar}

\begin{lemma}\label{L-nablaNN}
 Let $g_t\in\mathcal{M}$ and $\dt g_{t}=s_{t}\,\hat g_{t}$ for some $s_{t}\in C^1(M)$. Then
 the second fundamental tensor ${b^\perp}$ and its mean curvature vector ${H^\perp}$ are evolved as
\begin{equation}\label{E-nablaNNt2s}
 \dt {b^\perp} =-s\,{b^\perp},\qquad \dt{H^\perp} =-s\,{H^\perp}.
\end{equation}
\end{lemma}

\begin{proof}
We shall show for the more general setting $S=\dt g_t$ that
\begin{equation}\label{E-nablaNNt2}
 \dt {b^\perp} = -S^\sharp\circ{b^\perp},\qquad \dt {H^\perp} = -S^\sharp({H^\perp}).
\end{equation}
Using (\ref{eq2}), we compute for any $X\in D$ and $\eps_\alpha,\eps_\beta\in D^\perp$,
\begin{eqnarray*}
  g_t(\dt {b^\perp}(\eps_\alpha, \eps_\beta), X)
         \eq\frac 12\,g_t(\dt(\nabla^t_{\eps_\alpha} \eps_\beta) + \dt(\nabla^t_{\eps_\beta} \eps_\alpha),\ X)\\
 \eq\frac12\big[(\nabla^t_{\eps_\alpha} S)(X, \eps_\beta)+(\nabla^t_{\eps_\beta} S)(X, \eps_\alpha)
 -(\nabla^t_X S)(\eps_\alpha, \eps_\beta) \big]\\
 \eq-\frac12\big[S(\nabla^t_{\eps_\alpha} \eps_\beta, X)+S(\nabla^t_{\eps_\beta} \eps_\alpha, X)\big]
        =-S({b^\perp}(\eps_\alpha, \eps_\beta), X).
\end{eqnarray*}
From this (\ref{E-nablaNNt2})$_1$ follows when $S=s\,\hat g$.
Next, for any $X\in\Gamma(D)$, we have
\begin{eqnarray*}
 g_t(\dt{H^\perp}, X) \eq \sum\nolimits_\alpha g_t(\dt(\nabla_{\eps_\alpha}\eps_\alpha), X)
 = \frac12\big[2(\nabla_{\eps_\alpha} S)(\eps_\alpha, X) -(\nabla_X S)(\eps_\alpha,\eps_\alpha)\big]\\
  \eq-\sum\nolimits_\alpha S(\nabla_{\eps_\alpha}\eps_\alpha, X) = -S({H^\perp}, X),
\end{eqnarray*}
which confirms (\ref{E-nablaNNt2})$_2$. From (\ref{E-nablaNNt2}) for $S=s\,\hat g$ we obtain (\ref{E-nablaNNt2s}).
\end{proof}

Next we show that $D$-conformal variations of metrics preserve ``umbilical"
(i.e., ${b^\perp}={H^\perp}\cdot g_{|{D^\perp}}$) and ``harmonic" (i.e., ${H^\perp}=0$) properties of ${D^\perp}$.

\begin{proposition}\label{P-08bb}
Let $\dt g_t=s_t\,\hat g_t\ (s_t:M\to\RR)$, be a $D$-conformal family of Riemannian metrics on a manifold $(M,D,{D^\perp})$.
If $D^\perp$ is either umbilical (e.g., totally geodesic) or harmonic for $g_0$, then $D^\perp$ is the same for any $g_t$.
\end{proposition}

\begin{proof} If $D^\perp$ is $g_0$-umbilical, then $b^\perp=\frac1p\,H^\perp\,g^\perp$ at $t=0$,
where $H^\perp$ is the mean curvature vector field of $D^\perp$.
 Applying to (\ref{E-nablaNNt2s})$_1$ the theorem on existence and uniqueness of a solution of ODEs,
we conclude that $b^\perp_t=\frac1p\,\tilde H_t\,g^\perp_t$ for all $t$, where $\tilde H_t\in\Gamma(D)$.
Tracing this, we show that $\tilde H_t$ is the mean curvature vector of $b^\perp_t$, hence
$D^\perp$ is umbilical for any $g_t$. Indeed, $\tilde H_t\equiv0$ when $D^\perp$ is $g_0$-totally geodesic.
The remaining property, i.e., $D^\perp$ is harmonic, can be proved similarly.
\end{proof}

\section{Proofs of main results}
\label{sec:1-egf}

\subsection{Introducing the Extrinsic Geometric Flow}
\label{sec:1-def}

\begin{definition}\rm
Given $g_0=g$, a family of $(D,D^\perp)$-adapted Riemannian metrics $g_t,\ t\in[0,\eps)$, on $M$ will be called
(a) an \textit{Extrinsic Geometric Flow} (EGF) if (\ref{eq1}) holds;
(b) a \textit{normalized EGF} if
\begin{equation}\label{eq1n}
 \dt g_t = -\big(\,\frac2n\,\Div^\perp\!H_t +r(t)\big)\hat g_t,\quad {\rm where}\quad
 r(t)=-\frac2n\int_M (\Div^\perp\!H_t)\,d\vol_t/\vol(M,g_t).
\end{equation}
\end{definition}

The EGF (\ref{eq1}) and its normalized companion (\ref{eq1n}) provide some methods of evolving Riemannian metrics on foliated manifolds.  Obviously, both EGFs preserve harmonic (and totally geodesic) foliations.
If $g_0\in\mathcal{M}_1$, then all metrics $g_t\ (t\ge0)$ of (\ref{eq1n}) belong to $\mathcal{M}_1$,
because, see (\ref{E-dtdvol}),
\[
 \frac{d}{dt}\vol(M,g_t)=-\int_M\Big(\Div^\perp\!H_t +\frac n2\,r(t)\Big) d\vol_t
 = r(t)\vol(M,g_t)-\int_M r(t)\,d\vol_t =0.
\]
Substituting (\ref{E-divHt}) in the definition (\ref{eq1n}) of $r(t)$, we have
\[
 r(t)=-\frac2n\int_M g(H_t, H_t)\,d\vol_t\!/\vol(M,g_t)\le0.
\]

By~Proposition~\ref{P-08b}, EGFs preserve the following properties of $D^\perp$:
 (i)~{umbilical} $({b^\perp}={H^\perp}\cdot\,g_{|{D^\perp}})$,
 (ii)~{totally geodesic} $({b^\perp}=0)$,
 (iii)~harmonic $({H^\perp}=0)$;
which in case of integrable $D$ mean that the foliation tangent to $D$ is
(i)~{conformal}, (ii)~{Riemannian}, (iii)~transversally harmonic.

\vskip1mm
Let $g_t$ be a family of Riemannian metrics of finite volume on $(M,D,D^\perp)$.
Metrics $\tilde g_t = (\phi_t\hat g_{t})\oplus g_t^\perp$ with $\phi_t=\vol(M,g_t)^{-2/n}$ have unit volume:
$\int_M d\,\widetilde{\vol}_t=1$.
 The next proposition shows that \textit{unnormalized and normalized EGFs differ only by rescaling along the distribution~$D$}.

\begin{proposition}\label{P-normEG}
Let $g_t$ be a solution (of finite volume) to (\ref{eq1}) on $(M,D,D^\perp)$.
Then the metrics
\[
 \tilde g_{\,t}=(\phi_t\,\hat g_t)\oplus g^\perp,\quad
 {\rm where}\quad
 \phi_t=\vol(M,g_t)^{-2/n},
\]
evolve according to the normalized EGF
\begin{equation}\label{eq1-nB}
 \partial_{\,t}\,\tilde g_t = -\big(\frac2n\Div^\perp\tilde H_t+{\rho_t}\big)\,\hat{\tilde g}_t,\quad
 \mbox{ where }\ \rho_t={-\frac2n\int_M(\Div^\perp H)\,d\vol_t}\,/\,{\vol(M,g_t)}.
\end{equation}
\end{proposition}

\begin{proof}
Since $\phi$ depends only on $t$, by Lemma~\ref{L-btAt}, $\tilde H_t=H_t$ for metrics $\tilde g_t$ and $g_t$.
Hence, $\Div^\perp\tilde H_t=\Div^\perp H_t$.
From (\ref{E-dtdvol}) with $s=-\frac2n(\Div^\perp H_t)$ we get the derivative of the volume function
\[
 \frac{d}{dt}\vol(M,g_t)= \frac{d}{dt} \int_M d\vol_t = -\int_M (\Div^\perp H_t)\,d\vol_t.
\]
Thus $\phi_t=\vol(M,g_t)^{-2/n}$ is a smooth function of variable $t$.
By Lemma~\ref{L-btAt}, we have $\tilde b_t=\phi_t\cdot b_t$. Therefore
\[
 \partial_{\,t}\,\tilde g_t =\phi_t\dt g_t + \phi\,'_t\,\hat g_t = -\big(\frac2n\Div^\perp \widetilde H_t -\frac{\phi\,'_t}{\phi_t}\big)\,\hat{\tilde g}_t.
\]
Notice that $d\,\widetilde{\vol}_t=\phi^{n/2}_t\,d\vol_t$.
Using this and (\ref{E-dtdvol}), we obtain
\[
 \frac{d}{dt}\,\widetilde{\vol}_t=\frac{d}{dt}\,(\phi^{\frac n2}_t \vol_t)
 =\big(\frac n2\,\phi_t^{\frac n2-1}\phi'_t +\frac12\,\phi^{\frac n2}_t n s \big)\vol_t
  =\frac n2\Big(\frac{\phi'_t}{\phi_t}+s\Big)\widetilde{\vol}_t.
\]
Let $\rho_t$ be the average of $s$, see (\ref{eq1-nB}). From the above we get
\begin{eqnarray*}
 0\eq 2\,\frac{d}{dt}\int_M d\,\widetilde{\vol}_t =\int_M n\Big(\frac{\phi'_t}{\phi_t}+s\Big)d\,\widetilde{\vol}_t
 = n\big(\frac{\phi\,'_t}{\phi_t}+\rho_t\big).
\end{eqnarray*}
This shows that $\rho_t = -{\phi\,'_t}/{\phi_t}$.
Hence, $\tilde g$ evolves according to (\ref{eq1-nB}).
\end{proof}

\subsection{Proof of Theorems~\ref{T-02}--\ref{T-03} and Corollaries~\ref{C-umbilic}--\ref{C-minimal}}

\begin{proposition}\label{P-tauAk-H}
Let $D^\perp$ be integrable. The mean curvature vector $H_t$ of $D$ and its $D^\perp$-divergence
with respect to $g_t$ of EGF (\ref{eq1}) or (\ref{eq1n}) satisfy the following PDEs on any leaf of $D^\perp$:
\begin{equation}\label{E-tauAk-H-2}
 \dt H =\nabla^\perp(\Div^\perp \!H),
 \qquad
 \dt(\Div^\perp H) =\Delta^\perp(\Div^\perp H).
\end{equation}
\end{proposition}

\begin{proof}
By Lemma \ref{L-btAt2} with $s=-\frac2n\,(\Div^\perp H)$ or $s=-\frac2n\,\Div^\perp H-r(t)$,
we obtain (\ref{E-tauAk-H-2})$_1$. Similarly,  from (\ref{E-S-H})$_2$ we deduce (\ref{E-tauAk-H-2})$_2$.
\end{proof}

The eigenvalue problem, $-\Delta^\perp\,u=\lambda\,u$, on a leaf $L^\perp$ (of $D^\perp$ when it is integrable) has solution with a sequence of eigenvalues with repetition (each one as many times as the dimension of its finite-dimensional eigenspace) $0=\lambda_0<\lambda_1\le\lambda_2\le\cdots\uparrow\infty$.
 Let $\phi_j$ be an eigenfunction with eigenvalue $\lambda_j$ satisfying $\int_{L^\perp}\phi_j^2(x)\,d x=1$.
Then $G^\perp(t,x,y)=\sum_{j} e^{\lambda_j t}\phi_j(x)\,\phi_j(y)$ is a fundamental solution of the heat equation on $L^\perp$,
see details in Section~\ref{subsec:heat}.
 A solution satisfying $u(x,0)=u_0(x)$ is given by $u(x,t)=\int_{L^\perp} G^\perp(t,x,y) u_0(y)\,d y$.
Moreover, if $u_0\in L^2(L^\perp)$ then the solution converges uniformly, as $t\to\infty$, to a $D^\perp$-harmonic function (constant $\lambda_1$ when $L^\perp$ is closed).

\vskip1.5mm\noindent\textbf{Proof of Theorem~\ref{T-02}}.
This is based on the heat flow for 1-forms, see Section~\ref{sec:2.4-heat1}.
Let~$\theta^t_H$ be the dual 1-form on $D^\perp$ to the mean curvature vector field $H_t$
(with respect to $g_t$ for $t\ge0$).
Using $\Div^\perp H_t=-\delta^\perp\theta^t_H$ and
applying (\ref{E-S-H})$_3$ with $s=-\frac2n\Div^\perp H$, we show similarly to (\ref{E-tauAk-H-2})$_1$ that
\begin{equation}\label{E-tauAk-H}
 \dt\theta^t_H=-d^\perp\delta^\perp\theta^t_H,
\end{equation}
where $\theta^0_H=\theta_H$ is known.
We obtain (for $L_2$-product along the leaves of $D^\perp$)
\[
 \frac12\,\dt\,(\|\theta^t_H\|^2)=\dt\<\theta^t_H, \ \theta^t_H\>
 =-\<d^\perp\delta^\perp\theta^t_H,\ \theta^t_H\> =-\|\delta^\perp\theta^t_H\|^2\le0.
\]
By the above, we have uniqueness of a solution of the linear PDE (\ref{E-tauAk-H}).

 By Theorem~B in Section~\ref{sec:2.4-heat1}, the heat equation (considered on the leaves of $D^\perp$)
\begin{equation}\label{E-tauAk-thetaH-2}
 \dt\,\omega =\Delta^\perp_d\,\omega,\qquad\mbox{where}\quad\omega_0 =\theta_H,
\end{equation}
 admits a unique solution $\omega_t\ (t\ge0)$.
 As $t\to\infty$, the 1-form $\omega_t$ converges exponentially to a $D^\perp$-harmonic 1-form $\omega_\infty$,
 i.e., $\|\omega_t - \omega_\infty\|\le c\,e^{-\lambda\,t}$ for some constants $c,\lambda>0$.
 Since $\omega_0=\theta_H$ is $D^\perp$-closed, again by Theorem~B, we have $d^\perp\omega_t=0$ for all $t\ge0$.
 (Notice that for $p=1$ the 1-form $\theta_H$ is always $D^\perp$-closed).
 Comparing (\ref{E-tauAk-H}) with (\ref{E-tauAk-thetaH-2}), we conclude that $\theta^t_H=\omega_t$
 is a unique solution of (\ref{E-tauAk-H}).
 By~the above, $\Div^\perp H_\infty=0$ (here $H_\infty\in\Gamma(D^\perp)$ is dual to $\theta^\infty_H=\omega_\infty$).
 Applying the integral formula (\ref{E-divHt}), we conclude that
 $\lim\limits_{t\to\infty} \|H_t\|^2=0$, hence $H_\infty=0$.

 With known $H_t$, the PDE (\ref{eq1}) also has a unique smooth global solution, and
\[
 \hat g_t=\hat g_0\exp\big(\!-\frac2n\int_0^t (\Div^\perp H_s)\,ds\big).
\]
By~Lemma~\ref{L-nablaNN}, we also have
 $b^\perp_t=b^\perp_0\exp(\frac2n\int_0^t (\Div^\perp H_s)\,ds)$ for $t\ge0$.
Since the leaves of $D^\perp$ are compact, and $\Div^\perp H_t$ satisfies
(\ref{E-tauAk-H-2})$_2$ (on any leaf of $D^\perp$),
we have $|\Div^\perp H_t|\le e^{-\lambda_1 t}|\Div^\perp H_0|$~and
\[
 \Big|\int_0^t (\Div^\perp H_s)\,ds\Big|\le\int_0^t\big|\Div^\perp H_s\big|\,ds
 <\tilde c\int_0^t e^{-\lambda_1 t}\,ds =\tilde c\,\frac{1-e^{-\lambda_1 t}}{\lambda_1}<\frac{\tilde c}{\lambda_1}
\]
for some $\tilde c>0$ and all $t\ge0$. Now, for some $c\ge1$ and all $t\ge0$, we have the uniform bounds
\[
 c^{-1}\hat g_0\le \hat g_t\le c\,\hat g_0,\qquad
 c^{-1}b^\perp_0\le b^\perp_t\le c\,b^\perp_0.
\]
From the above and Proposition~\ref{P-converge},
$g_t\to g_\infty$ and $b^\perp_t\to b^\perp_\infty$ as $t\to\infty$ in $C^\infty$-topology.
\qed

\vskip1.5mm\noindent\textbf{Proof of Corollary~\ref{C-umbilic}}.
By Theorem~\ref{T-02}, the EGF metrics $g_t$ converge to a smooth metric $g_\infty$ with $H_\infty=0$.
By~Proposition~\ref{P-08b}, EGFs preserve the umbilicity of~$D$, hence $D$ is $g_\infty$-umbilical.
Note that an umbilical foliation with vanishing mean curvature is totally geodesic.
\qed

\vskip1.5mm\noindent\textbf{Proof of Corollary~\ref{C-minimal}}.
The leaves of orthogonal foliation are immersed compact cross-sections, hence $\calf$ is taut. On the other hand,
by Theorem~\ref{T-02}, the EGF metrics $g_t$ converge  as $t\to\infty$ to a smooth Riemannian metric $g_\infty$ with $H_\infty=0$.
\qed

\vskip1.5mm\noindent\textbf{Proof of Theorem~\ref{T-03}}.
The vector field $\widetilde H_t=H_t-X$ satisfies PDEs of Proposition~\ref{P-tauAk-H},
\begin{equation}\label{E-tauAk-H-X}
 \dt\widetilde H_t = \nabla^\perp\Div^\perp\widetilde H_t,\qquad
 \dt(\Div^\perp\widetilde H_t) =\Delta^\perp(\Div^\perp\widetilde H_t).
\end{equation}
Since $\dt X=0$, the 1-form $\theta^t_{\widetilde H}:=
 \theta_{H_t}-\theta_X$ satisfies the PDE, see (\ref{E-tauAk-H}),
\begin{equation}\label{E-tauAk-H-X2}
 \dt\theta_{\widetilde H}=-d^\perp\delta^\perp\theta_{\widetilde H}
\end{equation}
with the initial value  $\theta^0_{\widetilde H} =\theta_H-\theta_X$.
 As in the proof of Theorem~\ref{T-02}, we have uniqueness of a solution of
 (\ref{E-tauAk-H-X2}) for $t\ge0$.
By Theorem~B in Section~\ref{sec:2.4-heat1}, the heat equation (considered on the leaves of $D^\perp$)
\begin{equation}\label{E-tauAk-thetatildeH-0}
 \dt\,\omega =\Delta^\perp_d\,\omega
\end{equation}
admits a unique smooth solution $\omega_t\ (t\ge0)$ with the initial value $\omega_0 =\theta_H-\theta_X$.
 As $t\to\infty$, the 1-form $\omega_t$ converges exponentially to a $D^\perp$-harmonic 1-form $\omega_\infty$,
 and $\|\omega_t-\omega_\infty\|\le c\,e^{-\lambda\,t}$ for some constants $c,\lambda>0$.
 Since $\omega_0$ is $D^\perp$-closed, again by Theorem~B, we have $d^\perp\omega_t=0$ for all $t\ge0$.
 Comparing (\ref{E-tauAk-H-X2}) with (\ref{E-tauAk-thetatildeH-0}), as in the proof of Theorem~\ref{T-02}, we conclude that $\theta^t_{\widetilde H}=\omega_t$ is a unique solution of (\ref{E-tauAk-H-X2}).
 By the above, $\Div^\perp(H_\infty-X)=0$ (here $H_\infty-X$ is dual to $\omega_\infty$).
 With known $H_t$, the PDE (\ref{eq1-X}) also has a unique smooth global solution, and
 $\hat g_t=\hat g_0\exp\big(\!-\frac2n\int_0^t \Div^\perp(H_s-X)\,ds\big)$.
By~Lemma~\ref{L-nablaNN}, we also have
 $b^\perp_t=b^\perp_0\exp(\frac2n\int_0^t\Div^\perp(H_s-X)\,ds)$ for all $t\ge0$.
 Using $|\Div^\perp (H_t-X)|\le e^{-\lambda_1 t}|\Div^\perp(H_0-X)|$, we obtain
\[
 \Big|\int_0^t (\Div^\perp(H_s-X))\,ds\Big|\le\int_0^t\big|\Div^\perp(H_s-X)\big|\,ds
 <\tilde c\int_0^t e^{-\lambda_1 t}\,ds =\tilde c\,\frac{1-e^{-\lambda_1 t}}{\lambda_1}<\frac{\tilde c}{\lambda_1}
\]
for some $\tilde c>0$ and all $t\ge0$.
 As in the proof of Theorem~\ref{T-02}, we conclude that
 $g_t$ converges in $C^\infty$-topology as $t\to\infty$ to a Riemannian metric $g_\infty$
with $D^\perp$-harmonic 1-form $\theta_{H_\infty-X}$.

Certainly, if $H^1(L^\perp,\RR)=0$ for the leaves of $D^\perp$, then $H_\infty=X$.
\qed

\section{More examples}
\label{sec:more}

\subsection{The codimension-one case}

Let $(M,g)$ be a closed Riemannian manifold with a codimension-one distribution $D$ (i.e., $p=1$).
Let $N$ be the unit vector field (orthogonal to $D$), and $b$ the scalar second fundamental form of $D$
with respect to $N$. Indeed, $2\,b(X,Y)=g(\nabla_XY+\nabla_YX,\ N)$.
Hence,  $H=\tau_1\,N$ and $\tau_1=g(N, H)=\tr\!A$, where $A=b\,^\sharp$ is the shape operator.
By~Lemma~\ref{L-btAt2}, we find the variations (see also \cite{rw-m})
\begin{equation}\label{E-Atau1}
 \dt A = - \frac12\,N(s)\,\hat{\id},\qquad
 \dt\tau_1= -\frac n2\,N(s),
\end{equation}
where $\dt g_{t}=s_{t}\,\hat g_{t}$.
 For a codimension-one case, the EGF definitions (\ref{eq1}) and (\ref{eq1n}) read as:
\begin{eqnarray}\label{eq1-1}
 \dt g_t \eq -\frac2n\,N(\tau_1)\,\hat g_t,\\
\label{eq1n-1}
 \dt g_t \eq -\big(\,\frac2n\,N(\tau_1)+r(t)\big)\hat g_t,\qquad r(t)=-\frac2n\int_M N(\tau_1)\,d\vol_t/\vol(M,g_t),
\end{eqnarray}
where $g_0=g$.
The equality (\ref{E-divN}) with $\xi = f N$ reduces to the known formula, see \cite{rw-m},
\begin{equation}\label{E-IF-NF}
 \int_M N(f)\,d\vol =\int_M \tau_1(N)f\,d\vol.
\end{equation}
By the above, we have $ r(t)=-\frac2n\int_M \tau^2_1\,d\vol_t\!/\vol(M,g_t)\le0$.

\begin{proposition}\label{C-2ntau1}
Let $N$-curves compose a fibration $S^1\overset{i}\hookrightarrow M\overset{\pi}\to B$
of a closed Riemannian manifold $(M,g)$.
Then the following PDE:
\begin{equation*}
 \dt g_t = -\frac2n\,N(\tau_1)\,\hat g_t,\qquad g_0=g
\end{equation*}
 has a unique smooth global solution $g_t\ (t\ge0)$.
The metrics $g_t$ approach as $t\to\infty$ to a smooth metric $g_\infty$ for which $D$ is $g_\infty$-harmonic.
If $D$ is $g$-umbilical then it is $g_\infty$-totally geodesic.
\end{proposition}

\noindent\textbf{Proof}. For the EGF (\ref{eq1-1}) or (\ref{eq1n-1}),
 by (\ref{E-Atau1}) with $s=-\frac2n N(\tau_1)$ or $s=-\frac2n N(\tau_1)-r(t)$,
 both functions, $\tau_1$ and $N(\tau_1)$, satisfy the heat equation along $N$-curves:
\begin{equation*}
 \dt\tau_1= N(N(\tau_{1})),\qquad \dt(N(\tau_1))= N(\dt\tau_{1})= N(N(N(\tau_{1}))).
\end{equation*}
The unique solution, $\tau^t_1$, exists for all $t\ge0$, and $\tau_1^t\to\tau^\infty_1$ as $t\to\infty$.
Hence, $N(\tau^\infty_1)=const$ on any $N$-curve.
In case of (\ref{eq1n-1}), we have $r(t)\to0$ as $t\to\infty$.
Since $N$-curves are closed,
we obtain $N(\tau^\infty_1)=0$.
With known $\tau^t_1$, (\ref{eq1-1}) has a unique smooth solution
 $\hat g_t=\hat g_0\exp(-\frac2n\int_0^tN(\tau^s_1)\,ds)$ for all $t\ge0$.
We have the uniform bounds
 $c^{-1}\hat g_0\le \hat g_t\le c\,\hat g_0$,
 and
 $c^{-1}b^\perp_0\le b^\perp_t\le c\,b^\perp_0$
for some $c\ge1$ and all $t\ge0$.
By Proposition~\ref{P-converge} we obtain converging $g_t\to g_\infty$ (to a smooth metric) and $b^\perp_t\to b^\perp_\infty$
as $t\to\infty$.
 By the integral formula $\int_M [N(\tau^\infty_1)-(\tau^\infty_1)^2]\,d\,{\vol}_\infty=0$, see
 (\ref{E-IF-NF}) with $f=\tau_1$, we conclude that $\tau^\infty_1=0$ on $M$.
\qed

\subsection{Diffusion of double-twisted product metrics}
\label{sec:twisted}

 Let $M=M_1\times M_2$ be the product of closed Riemannian manifolds $(M_1,g_1)$ and $(M_2,g_2)$
 with the canonical projections $\pi_i:M\to M_i$.
 Given positive differentiable functions $f_i:M\to \RR\ (i = 1,2)$,
 the \textit{metric of a double-twisted product} $M_1\times_{(f_1,f_2)} M_2$ is defined by
 $g=(f_1^2 g_1)\oplus(f_2^2 g_2)$, i.e.,
\[
  g(X,Y) = f_1^2 g_1(\pi_{1*}X, \pi_{1*}Y) + f_2^2 g_2(\pi_{2*}X, \pi_{2*}Y),\quad X,Y\in TM.
\]
The foliations $M_1\times\{y\}$ and $\{x\}\times M_2$ are umbilical with
\textit{mean curvature vectors} $H_1=-\pi_{2*}\nabla(\log f_1)$ and $H_2=-\pi_{1*}\nabla(\log f_2)$.
 This property characterizes the double-twisted product,~\cite{pr}.
If $f_2=1$ then we have a \textit{twisted product}, and a \textit{warped product} if also $f_1$ depends on $M_2$ only.
By Proposition~\ref{P-08b}, EGFs (for $D=\ker\pi_{2*}$ and $D^\perp=\ker\pi_{1*}$) preserve the double-twisted product structure.

\vskip1mm
From Theorem~\ref{T-02} we deduce

\begin{proposition}\label{C-twisted}
Let $M_1\times_{(f_1,f_2)} M_2$ be a double-twisted product of closed Riemannian manifolds $(M_i, g_i)$.
Then (\ref{eq1}) has a unique smooth solution $g_t\in\mathcal{M}$ for all $t\ge0$, consisting of double-twisted product metrics on $M_1\times_{(f_1(t),f_2)} M_2$. As $t\to\infty$,
the metric $g_t$ converges in $C^\infty$-topology to a double-twisted product metric $g_\infty$ corresponding to
$M_1\times_{(\bar f_1,f_2)} M_2$, where $\bar f_1(x)=\int_{M_2} f_1(0, x, y)\,d y_g$.
If~$f_2=1$ (i.e., $g$ is a twisted product) then $g_\infty$ splits as the product $(M_1, \bar f_1\cdot g_1)\times(M_2,g_2)$.
\end{proposition}

\begin{proof}
Define $\phi:M\to\RR$ by $f_1=e^{\,\phi}$.
The mean curvature vector of $M_1\times\{y\}$ is $H=-\nabla^\perp\phi$.
The evolution $\dt g(t)= (s(t)\,g_1(t))\oplus 0$ preserves the double-twisted product structure.
Denoting $\hat g_t=g_1(t)$, we obtain $g(t)=(f_1(t)^2\hat g_1)\oplus(f_2^2g_2)$,
where $f_2$ and $g_2$ do not depend on $t$.
Assuming $f_1(t)=e^{\phi(t)}$, we find $\hat g_t=e^{2\phi(t)}\hat g_0$.
In this case, $\dt g(t)=2 e^{2\phi(t)}\dt\phi(t)\hat g_0$. Hence, $\dt\phi(t)=\frac12\,s(t)$.

 For $s=-\frac2n\Div^\perp H$, we have $\Div^\perp H=-\Delta^\perp\phi$, and (\ref{eq1}) reads as
\begin{eqnarray}\label{E-polynin1}
 \dt\phi \eq (1/n)\,\Delta^\perp \phi.
\end{eqnarray}
Replacing $t\to t/n$, one may reduce (\ref{E-polynin1}) to the heat equation.
Hence, see Proposition~\ref{P-normEG}, there is a unique solution $\phi(t,x,y)\ (t\ge0)$,
 satisfying $\lim\limits_{t\to\infty}\phi(t,x,y)=\bar\phi(x)$. Certainly, $\bar f_1=e^{\bar\phi}$.
\end{proof}

\begin{example}\rm
(a) Let $M=M_1\times M_2$ be the product of smooth manifolds, and $M_2$ admits a Riemannian metric $g_2$ of quasi-positive Ricci curvature (say, $M_2=S^p$). Then for any vector field $X$ on $M$,
satisfying the condition $d_{2}\theta_{H_0-X}=0$ (derivation along $M_2$),
there is a double-twisted product structure $M=M_1\times_{(f_1,f_2)} M_2$
for which $X$ is the mean curvature vector of a foliation $M_1\times\{y\}$.
 To show this, for any Riemannian metric $g_1$ on $M_1$, consider the metric $g=g_1\oplus g_2$ on $M$,
and apply Proposition~\ref{C-twisted}.

\vskip1mm
(b)
Let $S^1\times_{(f,\,1)} S^1$ be a twisted product of two circles.
Then (\ref{eq1}) has a unique smooth solution $g_t\in\mathcal{M}$ for all $t\ge0$,
consisting of twisted product metrics on the torus~$T^2$.
As~$t\to\infty$, the metric $g_t$ converges in $C^\infty$-topology to the flat metric on the
(generally, non-equilateral) torus.
\end{example}

\section{Appendix: The heat equation}
\label{sec:1-prelim}

Following \cite{jj2011}, we shall briefly discuss the heat equation and the heat flow on 1-forms.

\subsection{The heat equation on a Riemannian manifold}
\label{subsec:heat}

The covariant derivative of the $(1,j)$-tensor $T$ is the $(1,j+1)$-tensor given by
\begin{eqnarray*}
 (\nabla T)(X, Y_1,\ldots, Y_j) \eq(\nabla_X T)(Y_1,\ldots, Y_j)\\
 \eq\nabla_X(T(Y_1,\ldots, Y_j)) -\sum\nolimits_{i\le j} T(Y_1,\ldots,\nabla_X Y_i,\ldots Y_j).
\end{eqnarray*}
If $X$ is a vector field (i.e., a $(1,0)$-tensor), then $\nabla X$ is a $(1,1)$-tensor satisfying $(\nabla X)(Y)=\nabla_Y X$.
If $T$ is a $(1,j)$-tensor field on $(M,g)$, the \textit{divergence} $\Div T$ is the $(0,j)$-tensor
\[
 \Div T(Y_1,\ldots, Y_j) =\sum\nolimits_{i\le n} g((\nabla_{e_i}T)(Y_1,\ldots, Y_j), e_i),
\]
where $\{e_i\}\ (i\le n)$ is a local orthonormal frame.
The \textit{(rough) Laplacian} of $T$ is the divergence of the gradient of the tensor:
 $\Delta T  = \Div(\nabla T)$.
 When applied to functions, we get
the Hessian ${\rm Hess}(u)=\nabla d\,u$ and the Laplacian $\Delta u=\Div(d u)=\tr\!{\rm Hess}(u)$.
If $T$ is a vector, the gradient is a covariant derivative which results in a $(1,1)$-tensor, and the divergence of this is again a vector.

We will briefly discuss the homogeneous \textit{heat equation} on a closed Riemannian manifold $(M, g)$,
\begin{equation}\label{E-heat}
 \dt u=\Delta\,u.
\end{equation}
The \textit{eigenvalue problem} $-\Delta u=\lambda\,u$ on $(M,g)$ has solution with a sequence of eigenvalues
with repetition (each one as many times as the dimension of its finite dimensional eigenspace) $0=\lambda_0<\lambda_1\le\lambda_2\le\cdots\uparrow\infty$.
The smallest eigenvalue is $\lambda_0=0$ whose eigenfunction (normalized to have $L^2$-norm one) is the constant $\phi_0=\vol(M,g)^{-1/2}$.
Let $\phi_j$ be an eigenfunction with eigenvalue $\lambda_j$, satisfying $\int_{M}\phi_j^2(x)\,d x_g=1$.
The fundamental solution of (\ref{E-heat}),
 $G(t, x, y)=\sum\nolimits_{j} e^{-\lambda_j t}\phi_j(x)\,\phi_j(y)$,
is called the \textit{heat kernel}. A solution of (\ref{E-heat}), satisfying $u(x,0)=u_0(x)$ is given~by
\begin{equation}\label{E-heat-sol}
  u(x,t)=\int_{M} G(t,x,y)\,u_0(y)\,d y_g.
\end{equation}
Moreover, if $u_0\in L^2(M)$, the solution converges uniformly, as $t\to\infty$, to a  constant function (harmonic function when $M$ is open). Since $\lim\limits_{t\to\infty} G(t,x,y) = 1/\vol(M,g)$, from (\ref{E-heat-sol}) follows
that the equilibrium ``temperature" is the average of $u_0$:
 $\lim\limits_{t\to\infty} u(x,t) = \int_{M} u_0(x)\,dx_g/\vol(M,g)$.

\begin{example}\label{Ex-S1}\rm
(a) Over (semi-)infinite spatial interval $\RR$ (and similarly over $\RR^p$ or a non-compact flat manifold),
all solutions of (\ref{E-heat})
are obtained from the fundamental solution for the Dirac delta function $u_0(x)=\delta_x(\xi)$,
which is the function $G(t,x,y)=\frac 1{(4\pi t)^{1/2}}\, e^{-(x-y)^2/(4\,t)}$.
For any function $u_0\in L^2(\RR)$, a unique solution of (\ref{E-heat}),
$u(t,x)=\int_\RR u_0(y)G(t,x,y)\,d y$, converges uniformly to a linear function, as $t\to\infty$.
Indeed, if $u_0$ is bounded then the linear function is constant.

(b) The eigenvalues of $-\Delta$ on a flat $n$-torus are the numbers
$\lambda_{l_1\cdots l_n}=l_1^2+\ldots+l_n^2$
with corresponding eigenfunctions $\phi_{l_1\cdots l_n}=\frac1{2\pi}\,e^{-i(l_1 x_1+\ldots+l_n x_n)}$.
Here $l_1,\cdots,l_n$ take all possible positive and negative integer values.
 For a circle $S^1$ of radius one, we have $\lambda_l=l^2$ and $\phi_l(x)=\frac1{\sqrt{2\pi}}e^{-i l x}$, where $l\in\ZZ\setminus\{0\}$.
The Cauchy's problem on a circle
\begin{equation}\label{E-Heat-circle}
  \dt u = \ddx u,\qquad u(0,\cdot)=u_0\in H^2(S^1)
\end{equation}
has a unique solution in the class of functions  $C([0,\infty),\,H^2(S^1))\cap C^1((0,\infty],\,L^2(S^1))$.
By Sobolev embedding theorem, $H^2(S^1)\subset C^1(S^1)$.
A unique solution of (\ref{E-Heat-circle}) has the property $u(t,\cdot)\in C^\infty(S^1)$ for all $t>0$.
Moreover, $u(t,\cdot)\to\bar u_0=\frac1{2\pi}\int_{S^1} u_0(x)\,dx$ as $t\to\infty$,
and $\|u(t,\cdot)-\bar u_0\|\le e^{-t}\|u_0-\bar u_0\|$.
\end{example}

\subsection{The heat flow on 1-forms}
\label{sec:2.4-heat1}

Let $\theta$ be a 1-form in some cohomology class of a closed oriented differentiable manifold $M$.
Recall that the set of equivalence classes of closed 1-forms is a finite-dimensional vector space over $\RR$, called the \textit{first de Rham cohomology group} and denoted by $H^1(M,\RR)$.
A 1-form $\theta$ is closed if $d\theta=0$, and $\theta$ is exact, if there is a function $f$ with $d f=\theta$.
 Because of $d\circ d=0$, exact forms are always closed.
The closed 1-forms $\theta_0,\theta_1$ are cohomologous if  $\theta_0-\theta_1$ is exact.
 The \textit{divergence operator} $\delta$ on forms is adjoint to $d$ w.\,r.\,t. $L^2$-product on $TM$, i.e., $(d\theta,\,\eta)=(\theta,\,\delta\eta)$ for any $q$-form $\eta$ and  any $(q-1)$-form $\theta$.
Notice that $\delta\,\theta=-\Div\theta$ for 1-forms.

 The \textit{Hodge Laplacian operator} acting on differential forms is defined as follows:
 $\Delta_d = -(d\delta+\delta d)$,
it is (formally) selfadjoint and non-positive.
The rough Laplacian and the Hodge Laplacian on scalar functions are the same as the Laplace-Beltrami operator.

 A 1-form $\theta$ is called \textit{harmonic} if $\Delta_d\,\theta=0$;
in this case, on a closed manifold we have $d\,\theta=0$ and $\delta\,\theta=0$.
 In a local chart, for a 1-form $\theta=\theta_i dx^i$, we have
\[
 d\,\theta =\sum\nolimits_{k}\frac{\partial\theta_i}{\partial x^k}\,dx^k\wedge dx^i,\quad
 \delta\,\theta =-\sum\nolimits_{k,l}g^{kl}(\frac{\partial\theta_k}{\partial x^l} -\Gamma^j_{kl}\,\theta_j),\quad
 \Delta_d\,\theta =\sum\nolimits_{k}\frac{\partial^2\theta_i}{(\partial x^k)^2}\,dx^i.
\]
 By Hodge Theorem,
\textit{every cohomology class in $H^1(M,\RR)$ contains precisely one harmonic 1-form}.
A 1-form $\theta_X$ (dual to the vector field $X$) is harmonic if and only if $\Div X=0$ and the $(1,1)$-tensor $\nabla X$ is symmetric.
Following \cite{jj2011}, we present the result of \cite{ML1951} in a convenient for us form

\vskip1.5mm\noindent\textbf{Theorem B.}
 \textit{
Let $\theta$ be a 1-form on closed oriented manifold $M$ of class $C^{\infty}$.
Then the \textit{heat flow}
\begin{equation}\label{E-theta-3}
 \dt\,\theta_t = \Delta_d\,\theta_t,\qquad \theta_{0}=\theta.
\end{equation}
admits a unique solution $\theta_t$
for all $t\ge0$.
 As $t\to\infty$, $\theta_t$ converges exponentially in $ C^{\infty}$-topology towards a~harmonic 1-form $\theta_\infty$, i.e.,
 $\|\theta_t-\theta_\infty\|\le c\cdot e^{-\lambda\,t}$
for some positive $c,\lambda$ ($\lambda$ is independent of~$\theta$).
 If $\theta$ is closed, then all $\theta_t$ are closed as well.}

\section*{Acknowledgment}
The first author
was supported by the Marie-Curie actions grant
No.~276919,
and is grateful to
the Institute of Mathematics of the Jagiellonian University for the hospitality during his stay in Krakow.
The second author was partially supported by the Polish NCN grant No.~N201\,606540.
The authors would like to thank Pawe\l~Walczak (University of \L\'od\'z)
and Vladimir~Sharafutdinov (Sobolev Institute of Mathematics at Novosibirsk) for helpful corrections concerning the manuscript.

\end{document}